\documentclass[11pt,a4paper,twoside]{article}
\usepackage{graphics}
\usepackage{amssymb}
\usepackage{amsmath}
\usepackage{mathrsfs}
\usepackage{makeidx}
\usepackage{color}
\usepackage{graphicx}
\usepackage{subfigure}
\usepackage{multicol}
\usepackage{multirow}
\usepackage{float}
\usepackage{booktabs}
\usepackage{epsfig}
\usepackage{multirow}
\usepackage{epstopdf}
\usepackage[colorlinks,
            linkcolor=red,
            anchorcolor=blue,
            citecolor=blue
            ]{hyperref} 
\usepackage{caption}
\usepackage{algorithm,algorithmic}

\usepackage{latexsym, cite, bm, amsthm, amsmath}
\usepackage{enumerate}
\usepackage{marginnote}
\usepackage{epstopdf}

\def \[{\begin{equation}}
\def \]{\end{equation}}
\def\R{\mathbb{R}}
\newtheorem{thm}{Theorem}[section]

\newtheorem{lem}{Lemma}[section]

\newtheorem{defn}{Definition}[section]

\newtheorem{rem}{Remark}[section]

\newtheorem{exam}{Example}[section]
\numberwithin{equation}{section}

\title{\bf A dynamical system based on projection operator for solving absolute value equations associated with second-order cone}

\usepackage[marginal]{footmisc}
\usepackage{authblk}
\author[a]{Cairong Chen\thanks{Supported partially by the National Natural Science Foundation of China (Grant No. 11901024) and  the Natural Science Foundation of Fujian Province (Grant No. 2021J01661). Email address: cairongchen@fjnu.edu.cn.}}
\author[b]{Dongmei Yu\thanks{Corresponding author. Supported partially by the Ministry of Education in China of Humanities and Social Science Project (Grand No. 21YJCZH204), the Natural Science Foundation of Liaoning Province (Grand Nos. 2020-MS-301, LJ2020ZD002) and the Youth Talent Entrustment Project of Liaoning Provincial Federation Social Science Circles (Grand No. 2022lslwtkt-069). Email: {yudongmei1113@163.com}.}}
\author[c]{Deren Han\thanks{Supported partially by the National Natural Science Foundation of China (Grant Nos. 12131004 and 11625105). Email: {handr@buaa.edu.cn}.}}
\author[a]{Changfeng Ma\thanks{Supported partially by the National Key Research and Development Program of China (No.
2019YFC0312003). Email address: macf@fjnu.edu.cn.}}
\affil[a]{School of Mathematics and Statistics, FJKLMAA and Center for Applied Mathematics of Fujian Province, Fujian Normal University, Fuzhou, 350007, P.R. China}
\affil[b]{Institute for Optimization and Decision Analytics, Liaoning Technical University, Fuxin, 123000, P.R. China}
\affil[c]{LMIB of the Ministry of Education, School of Mathematical Sciences, Beihang University, Beijing 100191, P.R. China}

\begin{document}
\date{\today}
\maketitle


\begin{quote}
{\bf Abstract:}
A new equivalent reformulation of the absolute value equations associated with second-order cone (SOCAVEs) is emphasised, from which a dynamical system based on projection operator for solving SOCAVEs is constructed. Under proper assumptions, the equilibrium points of the dynamical system exist and could be (globally) asymptotically stable. Some numerical simulations are given to show the effectiveness of the proposed method.

{\small

}

\end{quote}
\section{Introduction}\label{sec:intro}
The \emph{second-order cone} (SOC) in $\mathbb{R}^n$ is defined by
$$
\mathcal{K}^n = \left\{(x_1, x_2)\in \mathbb{R}\times \mathbb{R}^{n-1} : \|x_2\|\le x_1 \right\},
$$
where $\|\cdot\|$ denotes the Euclidean norm. If $n = 1$, let $\mathcal{K}^n$ represent the set of nonnegative reals $\mathbb{R}_+$. Moreover, a general SOC $\mathcal{K}\subset \mathbb{R}^n$ could be the Cartesian product of SOCs \cite{chen2005,chen2012,fukushima2001}, i.e.,
$$
\mathcal{K} = \mathcal{K}^{n_1}\times \cdots \times \mathcal{K}^{n_r}
$$
with $n_1,\cdots,n_r,r\ge 1$ and $n_1 + \cdots + n_r = n$. Without loss of generality, we focus on the case that $r=1$ because all the analysis can be carried over to the setting of $r>1$ according to the property of Cartesian product. For any $x=(x_1,x_2)\in \mathbb{R}\times \mathbb{R}^{n-1}$ and $y = (y_1,y_2)\in \mathbb{R}\times \mathbb{R}^{n-1}$, their \emph{Jordan product} is defined as \cite{chen2005,chen2012,fukushima2001}
$$
x \circ y = \left( \langle x, y\rangle, y_1x_2 + x_1 y_2\right)\in \mathbb{R}\times \mathbb{R}^{n-1},
$$
where $\langle\cdot, \cdot\rangle$ denotes the Euclidean inner product in $\mathbb{R}^n$. With this definition, the vector $|x|$ in SOC $\mathcal{K}^n$ is computed by
\begin{equation}\label{eq:dave}
|x| = \sqrt{x\circ x}.
\end{equation}

In this paper, we consider the problem of solving the absolute value equations associated with SOC (SOCAVEs) of the form
\begin{equation}\label{eq:socave}
    Ax - |x| - b=0
\end{equation}
with $A\in \mathbb{R}^{n\times n}$ and $b\in \mathbb{R}^n$. Unless otherwise stated, throughout this paper $|x|$ is defined as in~\eqref{eq:dave}. SOCAVEs~\eqref{eq:socave} is a special case of the generalized absolute value equations associated with SOC (SOCGAVEs)
\begin{equation}\label{eq:socgave}
    Ax + B|x| - b=0
\end{equation}
with $A,B\in \mathbb{R}^{m\times n}$ and $b\in \mathbb{R}^m$. To our knowledge, SOCGAVEs~\eqref{eq:socgave}
was formally introduced by Hu, Huang and Zhang \cite{hu2011} and further studied in \cite{nguyen2019,miao2017,huang2019,miao2022b} and the references therein. In addition, SOCAVEs~\eqref{eq:socave} is a natural extension of the standard absolute value
equations~(AVEs)
\begin{equation}\label{eq:ave}
Ax - |x| =b,
\end{equation}
meanwhile, SOCGAVEs~\eqref{eq:socgave} is an extension of the generalized absolute value equations (GAVEs)
\begin{equation}\label{eq:gave}
Ax - B|x| =b.
\end{equation}
In AVEs~\eqref{eq:ave} and GAVEs~\eqref{eq:gave}, the vector $|x|$ denotes the componentwise absolute value of the vector $x\in \mathbb{R}^n$. It is known that GAVE~\eqref{eq:gave} was first introduced by Rohn in \cite{rohn2004} and further investigated in \cite{mang2007,prok2009,hlad2018} and the references therein. Obviously, AVEs~\eqref{eq:ave} is a special case of GAVEs~\eqref{eq:gave}.

Over the past two decades,  AVEs~\eqref{eq:ave} and GAVEs~\eqref{eq:gave} have been widely studied because of its relevance to many mathematical programming problems, such as the linear complementarity problem (LCP), bimatrix games and others; see e.g. \cite{mang2007,mame2006,prok2009}. Hence, abundant theoretical results and numerical algorithms for both AVEs~\eqref{eq:ave} and GAVEs~\eqref{eq:gave} have been established. On the theoretical aspect, for instance, Mangasarian \cite{mang2007} shown that solving GAVEs~\eqref{eq:gave} is NP-hard; if GAVE~\eqref{eq:gave} is solvable, checking whether it has a unique solution or multiple solutions is NP-complete \cite{prok2009}. Moreover, various sufficient or necessary conditions on solvability and non-solvability of AVEs \eqref{eq:ave} and GAVEs \eqref{eq:gave} were discussed in \cite{mezzadri2020,wuli2018,mame2006,prok2009,hladik2022}. The latest trend is to investigate error bounds and the condition number of AVEs~\eqref{eq:ave} \cite{zamani2020}. On the numerical aspect, there are many algorithms for solving AVEs~\eqref{eq:ave} and GAVEs~\eqref{eq:gave}. For example, the Newton-type methods \cite{mang2009,crfp2016,caqz2011}, the SOR-like method \cite{kema2017}, the concave minimization methods \cite{mang2007a,zamani2021}, the exact and inexact Douglas-Rachford splitting methods \cite{chen2022} and others; see e.g. \cite{chen2021,yuch2020,abhm2018,maer2018,maee2017,yu2022} and the references therein.

We are interested in SOCAVEs~\eqref{eq:socave} and SOCGAVEs~\eqref{eq:socgave} not only because they are extensions of the standard ones, but also because they are equivalent with some LCPs associated with SOC (SOCLCPs), which have various applications in engineering, control and finance~\cite{hu2011,miao2017,miao2021}. Recently, some numerical methods and theoretical results have been developed for SOCAVEs~\eqref{eq:socave} and SOCGAVEs~\eqref{eq:socgave}. For the numerical side, Hu, Huang and Zhang \cite{hu2011} proposed a generalized Newton method for solving SOCGAVEs~\eqref{eq:socgave} (Here and in the sequel, we assume $m=n$). Then, Huang and Ma \cite{huang2019} presented some weaker convergent conditions of the generalized Newton method. Miao et al. \cite{miao2017} proposed a smoothing Newton method for SOCGAVEs~\eqref{eq:socgave} and a unified way to construct smoothing functions is explored in \cite{nguyen2019}. Huang and Li \cite{huang2022} proposed a modified SOR-like method for SOCAVEs~\eqref{eq:socave}. Miao et al. \cite{miao2022} suggested a Levenberg-Marquardt method with Armijo line search for SOCAVEs~\eqref{eq:socave}. For the theoretical side, Miao et al. \cite{miao2022b} studied the existence and nonexistence of solution to SOCAVEs~\eqref{eq:socave} and SOCGAVEs~\eqref{eq:socgave}. In addition, the unique solvability for SOCAVEs~\eqref{eq:socave} and SOCGAVEs~\eqref{eq:socgave} was also investigated in \cite{miao2022b}. Miao and Chen \cite{miao2021} investigated conditions under which the unique solution of SOCAVEs~\eqref{eq:socave} is guaranteed, which are different from those in \cite{miao2022b}. Hu, Huang and Zhang proved that SOCGAVEs~\eqref{eq:socgave} is equivalent to the following problem: find $x,y\in \mathbb{R}^n$ such that
\begin{equation}\label{eq:sochlcp}
\begin{array}{l}
Mx + Py = c,\\
x\in \mathcal{K}^n,\quad y\in \mathcal{K}^n,\quad \langle x, y\rangle=0,
\end{array}
\end{equation}
where $M,P\in \mathbb{R}^{n\times n}$ and $c\in \mathbb{R}^n$. However, the problem~\eqref{eq:sochlcp} is not a standard SOCLCP, which is in the form of
\begin{equation}\label{eq:soclcp}
z\in \mathcal{K}^\ell,\quad w = Nz + q \in \mathcal{K}^\ell,\quad \langle z, w\rangle=0,
\end{equation}
where $N\in \mathbb{R}^{\ell\times \ell}$ and $q\in \mathbb{R}^\ell$. Miao et al. \cite{miao2017} showed that SOCGAVEs~\eqref{eq:socgave} is equivalent to SOCLCP~\eqref{eq:soclcp} with
$$
N = \left[\begin{array}{ccc}
-I & 2I & 0\\
A & B-A & 0\\
-A & A-B &0\end{array}\right], \quad z = \left[\begin{array}{c}
2x_+\\
|x|\\
0\end{array}\right]\quad \text{and}\quad q = \left[\begin{array}{c}
0\\
-b\\
b\end{array}\right],
$$
where $x_+$ is the projection of $x$ onto SOC $\mathcal{K}^n$. Note that the dimension of the above matrix $N$ is three times the dimension of the matrix $A$ (or $B$) (i.e., $\ell = 3n$). More recently, Miao and Chen \cite{miao2021}, under the condition that $1$ is not an eigenvalue of $A$ or $N$, provided the equivalence of SOCAVEs~\eqref{eq:socave} to the standard SOCLCP~\eqref{eq:soclcp} without changing the dimension (i.e., $\ell = n$).

The goals of this paper are twofold: to highlight another equivalent reformulation of SOCAVEs~\eqref{eq:socave} and to present a dynamical system to solve SOCAVEs~\eqref{eq:socave}. Contrast to the numerical methods mentioned above, our method is from a continuous perspective. Our work here is inspired by recent studies on AVEs~\eqref{eq:ave} \cite{chen2021}.

The rest of this paper is organized as follows. In section \ref{sec:pre}, we state a few basic results on the SOC and the autonomous system, relevant to our later developments. An equivalent reformation of SOCAVEs~\eqref{eq:socave} and a dynamical system to solve it are developed in section \ref{sec:main}. Numerical simulations are given in section \ref{sec:numerical}. Conclusions are made in section \ref{sec:conclusion}.

\textbf{Notation.} We use $\mathbb{R}^{n\times n}$ to denote the set of all $n \times n$ real matrices and $\mathbb{R}^{n}= \mathbb{R}^{n\times 1}$. $I$ is the identity matrix with suitable dimension. The transposition of a matrix or a vector is denoted by $\cdot ^\top$. The inner product of two vectors in $\mathbb{R}^n$ is defined as $\langle x, y\rangle\doteq x^\top y= \sum\limits_{i=1}^n x_i y_i$ and $\| x \|\doteq\sqrt{\langle x, x\rangle} $ denotes the Euclidean norm of vector $x\in \mathbb{R}^{n}$. $\|A\|$ denotes the spectral norm of $A$ and is defined by the formula $\| A \|\doteq \max \left\{ \| A x \| : x \in \mathbb{R}^{n}, \|x\|=1 \right\}$. $\textbf{tridiag}(a, b, c)$ denotes a matrix that has $a, b, c$ as the subdiagonal, main diagonal and superdiagonal entries in the matrix, respectively.

\section{Preliminaries}\label{sec:pre}
In this section, we collect a few important results which lay the foundation of our later analysis.

We first recall some basic concepts and background materials regarding SOCs, which can be found in \cite{faraut1994,chen2005,chen2012,fukushima2001,alizadeh2003}.

For $x = (x_1,x_2)\in \mathbb{R}\times \mathbb{R}^{n-1}$, the \emph{spectral decomposition} of $x$ with respect to SOC is given by
\begin{equation}\label{eq:sd}
x = \lambda_1(x) u_x^{(1)} + \lambda_2(x) u_x^{(2)},
\end{equation}
where
\begin{align*}
\lambda_i(x) &= x_1 + (-1)^i \|x_2\|,\\
u_x^{(i)} &= \left\{\begin{array}{l}
\frac{1}{2}\left(1, (-1)^i\frac{x_2}{\|x_2\|}\right), \quad \text{if}\quad x_2\neq 0,\\
\frac{1}{2}\left(1,(-1)^i w\right), \quad \text{if}\quad x_2 = 0\end{array}\right.
\end{align*}
for $i = 1,2$ and $w$ is any vector in $\mathbb{R}^{n-1}$ with $\|w\|=1$. If $x_2\neq 0$, the spectral decomposition is unique. $\lambda_1(x)$ and $\lambda_2(x)$ are called the eigenvalues of $x$ and $\left\{u_x^{(1)}, u_x^{(2)} \right\}$ is called a \emph{Jordan frame} of $x$. It is known that $\lambda_1(x)$ and $\lambda_2(x)$ are nonnegative if and only if $x\in \mathcal{K}^n$. For any real-valued function $f:\mathbb{R}\rightarrow \mathbb{R}$, we define a function on $\mathbb{R}^n$ associated with $\mathcal{K}^n$ by
$$f(x) \doteq f\left(\lambda_1(x)\right) u_x^{(1)} + f\left(\lambda_2(x)\right) u_x^{(2)}$$
if $x\in \mathbb{R}^n$ has the spectral decomposition \eqref{eq:sd}. Then we have
\begin{equation}\label{eq:xabs}
|x| = \left\{\begin{array}{l}
\frac{1}{2}\left( |x_1 - \|x_2\|| + |x_1 + \|x_2\||, \left( |x_1 + \|x_2\|| - |x_1 - \|x_2\||\right)\frac{x_2}{\|x_2\|}\right), \quad \text{if}\quad x_2\neq 0,\\
(|x_1|, 0),\quad \text{if}\quad x_2 = 0.\end{array}\right.
\end{equation}

The projection mapping from $\mathbb{R}^n$ onto $\Omega\subset \mathbb{R}^n$, denoted by $P_{\Omega}$, is defined as
$$P_{\Omega}(x) \doteq \arg\min\{\|x-y\|:y\in \Omega\}.$$
A fundamental property of the projection from $\R^n$ onto $\Omega$ is the following. Given $u\in \mathbb{R}^n$ and a nonempty closed convex subset $\Omega$ of $\R^n$, $w$ is the projection of $u$ onto $\Omega$, i.e., $w=P_{\Omega}(u)$ if and only if
\[\label{pp}
\langle u-w, v-w\rangle \le 0,\quad \forall~ v\in\Omega.\]
Next, we recall the projection onto SOC $\mathcal{K}^n$. As mentioned earlier, let $x_+$ be the projection of $x = (x_1,x_2)\in \mathbb{R}\times \mathbb{R}^{n-1}$ onto $\mathcal{K}^n$, then the explicit formula of the projection is characterized in \cite{chen2005,chen2012,faraut1994,fukushima2001} as below:
\begin{equation}\label{eq:pj}
x_+ = \left\{\begin{array}{l}
x, \quad \text{if}\quad x\in \mathcal{K}^n,\\
0, \quad \text{if}\quad x\in -\mathcal{K}^n,\\
u, \quad \text{otherwise,}
\end{array}\right.
\quad \text{where}\quad
u = \left[\begin{array}{c}
\frac{x_1 + \|x_2\|}{2}\\
\left(\frac{x_1 + \|x_2\|}{2}\right)\frac{x_2}{\|x_2\|}
\end{array}\right].
\end{equation}

Before talking something about the autonomous system, we will give the definitions of the dual cone and the Lipschitz continuity.

\begin{defn}
The dual cone of $\mathcal{K}^n$ is defined as
$$
(\mathcal{K}^n)^* \doteq \{y\in \mathbb{R}^n:\langle x,y\rangle \ge 0, \forall x\in \mathcal{K}^n\}.
$$
\end{defn}
It is known that SOC $\mathcal{K}^n$ is a pointed close convex cone and it is self-dual (i.e., $(\mathcal{K}^n)^* = \mathcal{K}^n$).

\begin{defn}
The function $F:\mathbb{R}^n\rightarrow \mathbb{R}^n$ is said to be Lipschitz continuous with Lipschitz constant $L>0$ if
$$
\|F(x)-F(y)\|\le L\|x-y\|,\quad \forall x,y \in \mathbb{R}^n.
$$
\end{defn}

Consider the autonomous system
\begin{equation}\label{eq:auto-dysm}
\frac{dx}{dt} = g(x),
\end{equation}
where $g$ is a function from $\mathbb{R}^n$ to $\mathbb{R}^n$. Throughout this paper, denote $x(t;x(t_0))$ the solution of~\eqref{eq:auto-dysm} determined by the initial value condition $x(t_0) = x_0$. The following results are well-known and can be found in \cite[Chapter~2 and Chapter~3]{khalil1996}.

\begin{lem}\label{lem:solution}
Assume that $g:\mathbb{R}^n\rightarrow\mathbb{R}^n$ is a continuous function, then for arbitrary $t_0\ge 0$ and $x(t_0)=x_0\in \mathbb{R}^n$, there exists a local solution $x(t;(x(t_0)),\,t\in[t_0,\tau]$ for some $\tau> t_0$. Furthermore, if $g$ is locally Lipschitz continuous at $x_0$, then the solution is unique; and if $g$ is Lipschitz continuous in $\mathbb{R}^n$, then $\tau$ can be extended to $+\infty$.
\end{lem}

\begin{defn}$($Equilibrium point$)$. Let $x^*\in \mathbb{R}^n$, then it is called an equilibrium point of the dynamical system \eqref{eq:auto-dysm} if $g(x^*) = 0$.
\end{defn}

\begin{defn}
$($Stability in the sense of Lyapunov$)$. The equilibrium point $x^*$ of \eqref{eq:auto-dysm} is stable if, for any $\epsilon > 0$, there is a $\delta = \delta(\epsilon)>0$ such that
$$
\|x(t_0) - x^*\|< \delta \quad \Rightarrow \quad\|x(t;x(t_0)) - x^*\|<\epsilon, \forall t \ge t_0.
$$
Furthermore, the equilibrium point $x^*$ of \eqref{eq:auto-dysm} is asymptotically stable if it is stable and $\delta$ can be chosen such that
$$
\|x(t_0)-x^*\|< \delta \quad\Rightarrow \quad\lim_{t\rightarrow\infty}x(t;x(t_0))=x^*.
$$
\end{defn}

\begin{thm}\label{thm:ls}
$($Lyapunov's stability theorem$)$. Let $x^*$ be an equilibrium point of \eqref{eq:auto-dysm} and $\Omega \subseteq \mathbb{R}^n$ be a domain containing $x^*$. If there is a continuously differentiable function $V:\Omega \rightarrow \mathbb{R}$ such that
$$
V(x^*) = 0 \quad \text{and} \quad V(x)> 0,\;\forall x\in \Omega\backslash\{x^*\},
$$
$$
\frac{dV(x)}{dt}=\nabla V(x)^\top g(x)\le 0,\;\forall x\in \Omega,
$$
then $x^*$ is stable. Moreover, if
$$
\frac{dV(x)}{dt}< 0,\;\forall x\in \Omega\backslash\{x^*\},
$$
then $x^*$ is asymptotically stable.
\end{thm}

\begin{thm}\label{thm:als}
Let $x^*$ be an equilibrium point for \eqref{eq:auto-dysm}. Let $V:\mathbb{R}^n \rightarrow \mathbb{R}$ be a continuously differentiable function such that
$$
V(x^*) = 0 \quad \text{and} \quad V(x)> 0,\;\forall x\neq  x^*,
$$
$$\frac{dV(x)}{dt}< 0,\;\forall x\neq  x^*,
$$
$$
\|x-x^*\|\rightarrow \infty \quad\Rightarrow \quad V(x)\rightarrow\infty,
$$
then $x^*$ is globally asymptotically stable.
\end{thm}

\section{The equivalent reformulation and the dynamical system}\label{sec:main}
In this section, we will associate to the SOCAVEs~\eqref{eq:socave} a dynamical model and then carry out an asymptotic stability analysis for the equilibrium points. To this end, we first highlight that SOCAVEs~\eqref{eq:socave} is equivalent to the generalized SOCLCP (SOCGLCP) as follows:
\begin{equation}\label{eq:socglcp}
Q(x)\doteq Ax + x - b \in \mathcal{K}^n, \quad F(x)\doteq Ax - x - b \in \mathcal{K}^n,\quad \left\langle Q(x), F(x)\right\rangle = 0.
\end{equation}
Moreover, SOCGLCP~\eqref{eq:socglcp} is equivalent to the generalized linear variational inequality problem associated with SOC (SOCGLVI) \cite{facchinei2003}:
\begin{equation}\label{eq:socglvi}
\text{find an}~x^*\in \mathbb{R}^n,~ \text{such that}~Q(x^*)\in \mathcal{K}^n, ~ \left\langle v-Q(x^*),F(x^*)\right\rangle\ge 0,\quad \forall\, v\in \mathcal{K}^n,
\end{equation}
which is also equivalent to
\begin{equation}\label{eq:peq}
Q(x) = P_{\mathcal{K}^n}\left[Q(x) - F(x)\right].
\end{equation}

In order to claim the equivalence between SOCAVEs~\eqref{eq:socave} and SOCGLCP~\eqref{eq:socglcp}, we introduce the following two lemmas.

\begin{lem}\label{lem:ab}(\cite{facchinei2003,mang2007})
Let $a,b\in \mathbb{R}$. Then $a\ge 0$, $b\ge 0$ and $ab = 0$ if and only if $a+b = |a-b|$.
\end{lem}

\begin{lem}\label{lem:xy}
Let $s,t\in \mathbb{R}^n$. Then
\begin{equation}\label{eq:com}
s\in \mathcal{K}^n,\quad t\in \mathcal{K}^n \quad \text{and}\quad \langle s,t\rangle = 0
\end{equation}
if and only if
\begin{equation}\label{eq:pm}
s+t = |s-t|.
\end{equation}
\end{lem}
\begin{proof}
We first prove that \eqref{eq:com} $\Rightarrow$ \eqref{eq:pm}. Since $s=(s_1,s_2)\in \mathcal{K}^n$ and $t=(t_1,t_2)\in \mathcal{K}^n$, we have
$$s_1\ge \|s_2\|\quad \text{and} \quad t_1\ge \|t_2\|,$$
which implies that
$$|\langle s_2,t_2\rangle|\le \|s_2\|\|t_2\|\le s_1t_1.$$
Thus,
$$\langle s,t\rangle = s_1t_1 + s_2^\top t_2\ge s_1t_1-s_1t_1 = 0,$$
and the equality is valid if and only if $s_2=kt_2\,(k\ge 0)$, $s_1=\|s_2\|$ and $t_1=\|t_2\|$. Hence, the vectors $s$ and $t$ in \eqref{eq:com} share the same Jordan frame~\cite{miao2021,alizadeh2003}. Let $s = \lambda_1 e_1+\lambda_2 e_2$ and $t=\mu_1 e_1 + \mu_2 e_2$, where $\{e_1,e_2\}$ is the Jordan frame. Then we have $\lambda_i,\mu_i\ge 0$ for $i=1,2$ and $\lambda_1\mu_1 = \lambda_2\mu_2 = 0$. It then follows from Lemma~\ref{lem:ab} that $\lambda_1+\mu_1 = |\lambda_1-\mu_1|$ and $\lambda_2+\mu_2 = |\lambda_2-\mu_2|$. On the other hand, we have
$s+t = (\lambda_1+\mu_1)e_1 + (\lambda_2+\mu_2)e_2$ and $|s-t| = |\lambda_1-\mu_1|e_1 + |\lambda_2-\mu_2|e_2$. Hence, we have \eqref{eq:pm}.

Next, we prove that \eqref{eq:pm} $\Rightarrow$ \eqref{eq:com}. By \eqref{eq:pm}, we known that $s+t$ and $s-t$ have the same Jordan frame, from which we obtain that $s$ and $t$ have the same Jordan frame. Indeed, it follows from the fact that
$$2s = (s+t) + (s-t)\quad \text{and}\quad 2t = (s+t) - (s-t).$$
Let $s = \lambda_1 e_1+\lambda_2 e_2$ and $t=\mu_1 e_1 + \mu_2 e_2$. Then it follows from \eqref{eq:pm} that $\lambda_1+\mu_1 = |\lambda_1-\mu_1|$ and $\lambda_2+\mu_2 = |\lambda_2-\mu_2|$, which combine with Lemma~\ref{lem:ab} implies that $\lambda_i,\mu_i\ge 0$ for $i=1,2$ and $\lambda_1\mu_1 = \lambda_2\mu_2 = 0$. Then we can obtain \eqref{eq:com}.
\end{proof}

\begin{rem}
Lemma \ref{lem:xy} can be found in \cite[Proposition~4.1]{fukushima2001} and \cite[Proposition~2.3]{chen2003}. Here we give a new proof based on the Jordan frame, which is different from that of \cite[Proposition~4.1]{fukushima2001}.
\end{rem}

According to Lemma~\ref{lem:xy}, if we set $s+t = Ax-b$ and $s-t = x$, we can obtain the equivalence between SOCAVEs \eqref{eq:socave} and SOCGLCP~\eqref{eq:socglcp}.

\begin{rem}
We should point out that the equivalence between SOCGLCP~\eqref{eq:socglcp} and SOCAVEs~\eqref{eq:socave}
 is implicit in the proof of \cite[Theorem~4.1]{miao2021} and the proof of Lemma~\ref{lem:xy} is also inspired by that of \cite[Theorem~4.1]{miao2021}.
\end{rem}

Let
\begin{equation}\label{eq:res}
r(x) = Q(x) - P_{\mathcal{K}^n}\left[Q(x) - F(x)\right]
\end{equation}
be the residual of \eqref{eq:peq}, then we have the following theorem.

\begin{thm}\label{thm:equiv}
$x^*$ solves SOCAVEs~\eqref{eq:socave} if and only if $r(x^*)=0$. Furthermore, it can be proved that
\begin{equation}\label{eq:ex}
r(x) = Ax -|x| -b.
\end{equation}
\end{thm}
\begin{proof}
Note that $Q(x) - F(x) = 2x$. We will split the proof into three cases.

\begin{itemize}
  \item [(a)] If $x\in \mathcal{K}^n$, then $2x\in \mathcal{K}^n$. It follows from $|x| = x$ and $P_{\mathcal{K}^n}(2x) = 2x$ that
$$Q(x) - P_{\mathcal{K}^n}\left[Q(x) - F(x)\right] = Ax + x - b - 2x = Ax - x -b = Ax - |x| -b.$$

  \item [(b)] If $x\in -\mathcal{K}^n$, $2x\in -\mathcal{K}^n$. Then it follows from $|x| = -x$ and $P_{\mathcal{K}^n}(2x) = 0$ that
$$Q(x) - P_{\mathcal{K}^n}\left[Q(x) - F(x)\right] = Ax + x - b - 0 = Ax - (-x) -b = Ax - |x| -b.$$

  \item [(c)] If $x\notin \mathcal{K}^n$ and $x\notin -\mathcal{K}^n$, it follows from \eqref{eq:pj} that
\begin{equation*}
P_{\mathcal{K}^n}(2x) = \left[\begin{array}{c}
x_1 + \|x_2\|\\
\frac{x_1x_2}{\|x_2\|}+ x_2
\end{array}\right].
\end{equation*}
In addition, it follows from \eqref{eq:xabs} that
$$
|x| = \left[\begin{array}{c}
\|x_2\|\\
\frac{x_1 x_2}{\|x_2\|}
\end{array}\right].
$$
Then
$$
Q(x) - P_{\mathcal{K}^n}\left[Q(x) - F(x)\right] = Ax + x - b -  \left[\begin{array}{c}
x_1 + \|x_2\|\\
\frac{x_1x_2}{\|x_2\|}+ x_2
\end{array}\right] =  Ax -  \left[\begin{array}{c} \|x_2\|\\
\frac{x_1x_2}{\|x_2\|}
\end{array}\right] - b = Ax - |x| -b.
$$
\end{itemize}
The proof is completed.
\end{proof}

Based on Theorem~\ref{thm:equiv}, we will approach the solution set of SOCAVEs~\eqref{eq:socave} from a continuous perspective by means of trajectories generated by the following projection-type dynamical system:
\begin{equation}\label{eq:dymo}
\frac{dx}{dt} = \gamma A^\top\left\{P_{\mathcal{K}^n}\left[Q(x)-F(x)\right]-Q(x)\right\},
\end{equation}
where $\gamma>0$ is the convergence rate. The construction of the dynamical system \eqref{eq:dymo} is inspired by \cite{xia2004,liu2010,hu2007,gao2001}. Substituting \eqref{eq:res} and \eqref{eq:ex} into \eqref{eq:dymo}, it is reduced to
\begin{equation}\label{eq:dymos}
\frac{dx}{dt} = \gamma A^\top(b + |x| - Ax)\doteq h(x).
\end{equation}
When $A$ is nonsingular, it is easy to prove that the equilibrium points of \eqref{eq:dymos} equal the solutions to SOCAVEs~\eqref{eq:socave}. That is, we have the following theorem.
\begin{thm}\label{thm:equi}
Let $A$ be nonsingular, then $x^*$ is a solution of SOCAVEs~\eqref{eq:socave} if and only if $x^*$ is an equilibrium point of the dynamical system \eqref{eq:dymos}.
\end{thm}

Now we are in the position to study the existence of the solutions and the stability of the equilibrium points of the dynamical system \eqref{eq:dymos}. We first consider the existence of the solutions of the dynamical system \eqref{eq:dymos}.

\begin{lem}\label{lem:lipcon}
The function $h$ defined as in \eqref{eq:dymos} is Lipschitz continuous in $\mathbb{R}^n$ with Lipschitz constant $\gamma \|A^\top\|(\|A\|+1)$.
\end{lem}
\begin{proof}
For any $x_1,x_2\in \mathbb{R}^n$, we have
\begin{align*}
\|h(x_1)- h(x_2)\| &= \|\gamma A^\top (b + |x_1| - Ax_1)-\gamma A^\top (b + |x_2| - Ax_2)\|\\
&= \gamma\| A^\top \left[A(x_2-x_1)+(|x_1|-|x_2|)\right]\|\\
&\le \gamma \|A^\top\|(\|A\|+1)\|x_1-x_2\|,
\end{align*}
where $\||x_1|-|x_2|\|\le \|x_1-x_2\|$ \cite{huang2019,miao2022b,ke2018} is used in the last inequality. Thus $h$ is a Lipschitz continuous function in $\mathbb{R}^n$ with Lipschitz constant $\gamma \|A^\top\|(\|A\|+1)$.
\end{proof}

Then the following theorem is concluded from Lemma~\ref{lem:solution} and Lemma~\ref{lem:lipcon}.

\begin{thm}
For a given initial value $x(t_0)=x_0$, there exists a unique solution $x(t;x(t_0)),t\in [0,\infty)$ for the dynamical system \eqref{eq:dymos}.
\end{thm}

Finally, we consider the stability of the equilibrium points of the dynamical system \eqref{eq:dymos}. To this end we need the following theorem and its proof is inspired by that of \cite[Theorem~2]{he1999}.

\begin{thm}\label{thm:contrac}
If $x^*$ is a solution of the SOCAVEs \eqref{eq:socave} and $\|A^{-1}\|\le 1$, then
\begin{equation}\label{ie:contr}
(x-x^*)^\top A^\top r(x) \ge\frac{1}{2} \left\|r(x)\right\|^2,\; \forall \,x\in \mathbb{R}^n.
\end{equation}
\end{thm}
\begin{proof}
As mentioned earlier, $x^*$ is also a solution of SOCGLCP \eqref{eq:socglcp} or  SOCGLVI~\eqref{eq:socglvi}.

Since  $\mathcal{K}^n$ is a closed convex set and $Q(x^*)\in \mathcal{K}^n$, it follows from the property \eqref{pp} of the projection mapping that
\begin{align}\nonumber
[v-P_{\mathcal{K}^n}(v)]^\top[P_{\mathcal{K}^n}(v)-Q(x^*)]\ge 0,\quad \forall \;v\in \mathbb{R}^n.
\end{align}
By seting $v\doteq Q(x)-F(x)$, we have
\begin{align}\label{equ6}
[r(x)-F(x)]^\top \{P_{\mathcal{K}^n}[Q(x)-F(x)]-Q(x^*)\}\ge 0.
\end{align}
In addition, it follows from $P_{\mathcal{K}^n}(\cdot)\in \mathcal{K}^n$, $F(x^*)\in \mathcal{K}^n = (\mathcal{K}^n)^*$ and $Q(x^*)^\top F(x^*)=0$ that
\begin{align}\label{equ7}
F(x^*)^\top \{P_{\mathcal{K}^n}[Q(x)-F(x)]-Q(x^*)\}\ge 0.
\end{align}
Adding \eqref{equ6} to \eqref{equ7} and using
\begin{align}\nonumber
P_{\mathcal{K}^n}[Q(x)-F(x)]-Q(x^*)=[Q(x)-Q(x^*)]-r(x),
\end{align}
we obtain
\begin{align}\nonumber
&\left\{[Q(x)-Q(x^*)]+[F(x)-F(x^*)]\right\}^\top r(x)\\\nonumber
&\qquad \ge \left\|r(x)\right\|^2+[Q(x)-Q(x^*)]^\top[F(x)-F(x^*)].
\end{align}
Then it follows from the definitions of $Q$ and $F$ in \eqref{eq:socglcp} that
\begin{equation*}
2 r(x)^\top A(x-x^*)\ge \left\|r(x)\right\|^2+(x-x^*)^\top(A^\top A-I)(x-x^*),\quad \forall \;x\in \mathbb{R}^n,
\end{equation*}
which together with $\|A^{-1}\|\le 1$ completes the proof.
\end{proof}

Now we can give the following stability theorem.
\begin{thm}\label{thm:3-6}
Let $\|A^{-1}\|\le 1$, then the equilibrium point $x^*$ $($if it exists$)$ of the dynamical system \eqref{eq:dymos} is asymptotically stable. In particular, if $\|A^{-1}\|< 1$, then the unique equilibrium point $x^*$ of the dynamical system \eqref{eq:dymos} is globally asymptotically stable.
\end{thm}
\begin{proof}
The proof is similar to that of \cite[Theorem~3.6]{chen2021}, we still give it here for completeness.

Let $x = x(t;x(t_0))$ be the solution of \eqref{eq:dymos} with initial value $ x(t_0) = x_0 $ and $x^*$ is the equilibrium point nearby $x_0$. Consider the following Lyapunov function:
$$
V(x)= e^{\|x-x^*\|^2}-1, \quad x\in \mathbb{R}^n.
$$
It is obvious that $V(x^*)=0$ and $V(x)>0$ for all $x\neq x^*$. In addition, it follows from \eqref{ie:contr} that
\begin{align*}
\frac{d}{dt}V(x)&= \frac{dV}{dx}\frac{dx}{dt}\\
&= -2 \gamma e^{\|x-x^*\|^2}(x-x^*)^\top A^\top r(x)\\
&\le -\gamma e^{\|x-x^*\|^2}\left\|r(x)\right\|^2< 0,\quad \forall x\neq x^*.
\end{align*}
Hence, the first part of the theorem is now a direct consequence of Theorem~\ref{thm:ls}.

When $\|A^{-1}\|< 1$, SOCAVEs~\eqref{eq:socave} has a unique solution \cite{miao2022b} and thus the equilibrium point is unique. Then it follows from $V(x)\rightarrow \infty$ as $\|x-x^*\|\rightarrow\infty$ and Theorem~\ref{thm:als} that the unique equilibrium point is globally asymptotically stable.
\end{proof}

\begin{rem}
If $\|A^{-1}\| = 1$, then SOCAVEs~\eqref{eq:socave} may have no solutions, more than one solutions and a unique solution. See the next section for more details.
\end{rem}

\section{Numerical simulations}\label{sec:numerical}
In this section, we will present four examples to illustrate the effectiveness of the proposed method. All experiments are implemented in MATLAB R2018b with a machine precision $2.22\times 10^{-16}$ on a PC Windows 10 operating system with an Intel i7-9700 CPU and 8GB RAM. We use ``ode23" to solve the ordinary differential equation. Concretely, the MATLAB expression
$$
[t,y] = \text{ode23}(odefun,tspan,y_0)
$$
is used, which integrates the system of differential equations from $t_0$ to $t_f$ with $tspan = [t_0,t_f]$. Here, ``odefun'' is a function handle.

\begin{exam}[\!\!\cite{huang2019}]\label{ex:sorvsdrs} Consider SOCAVEs \eqref{eq:socave} with
$A = \textbf{tridiag}(-1,4,-1)\in \mathbb{R}^{n\times n}$ and $b=Ax^*-|x^*|,$
where $x^*=(-1,1,-1,1,\cdots,-1,1)^\top\in \mathbb{R}^n.$

In this example, we have $\|A^{-1}\|<1$ and thus SOCAVEs \eqref{eq:socave} has a unique solution for any $b\in\mathbb{R}^n$. Equivalently, the dynamical model \eqref{eq:dymos} has a unique equilibrium point and its globally asymptotical stability will be numerically checked. We set $t_0 = 0$ and $t_f = 0.1$ for this example. Firstly, Figure~\ref{fig:fig-for-2or3} shows the phase diagram of the state $x(t)$ with different initial points for $n=2$ and $n=3$, which visually display the globally asymptotical stability of the equilibrium point. Secondly, it is shown in Figure~\ref{fig:fig-for-err} that a larger $\gamma$ leads to a faster convergence rate, the same phenomenon occurred in some existing works, such as \cite{huang2016,chen2021}.

\begin{figure}[t]
{\centering
\begin{tabular}{ccc}
\hspace{-0.3 cm}
\resizebox*{0.50\textwidth}{0.26\textheight}{\includegraphics{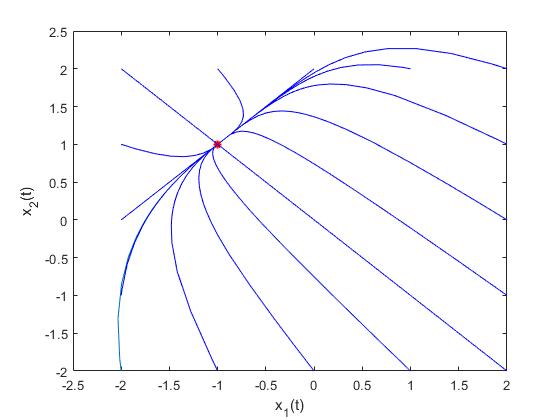}}
& & \hspace{-0.9 cm}
\resizebox*{0.50\textwidth}{0.26\textheight}{\includegraphics{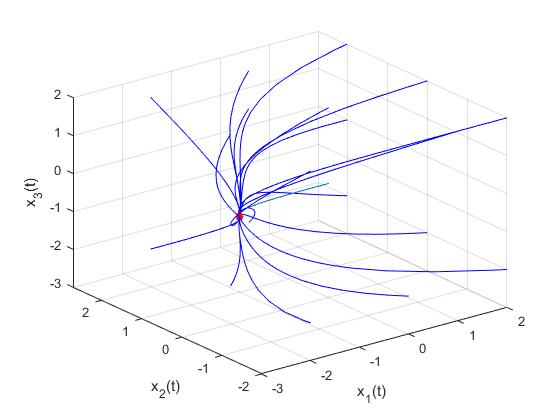}} \vspace{2ex}\\
\end{tabular}\par
}\vspace{-0.15 cm}
\caption{Phase diagrams of \eqref{eq:dymos} with different initial points for  Example~\ref{ex:sorvsdrs} (Left figure: $n=2$ and $16$ different initial points are used; Right figure: $n=3$ and $20$ different initial points are used). The red star point is the exact equilibrium point.}
\label{fig:fig-for-2or3}
\end{figure}

\begin{figure}[t]
{\centering
\begin{tabular}{ccc}
\hspace{-0.3 cm}
\resizebox*{0.50\textwidth}{0.26\textheight}{\includegraphics{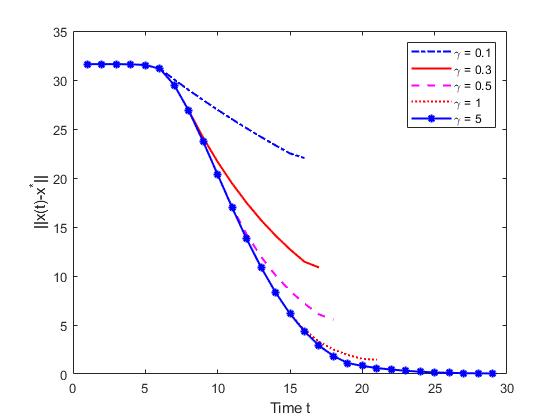}}
& & \hspace{-0.9 cm}
\resizebox*{0.50\textwidth}{0.26\textheight}{\includegraphics{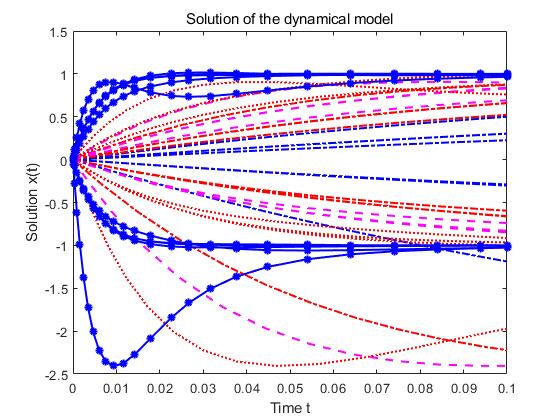}} \vspace{2ex}\\
\end{tabular}\par
}\vspace{-0.15 cm}
\caption{Convergence behaviors of $\|x(t)-x^*\|$ (left figure) and transient behaviors of $x(t)$ (right figure) for Example~\ref{ex:sorvsdrs} with $n=1000$ and $x_0=(0,0,\cdots,0)^\top$.}
\label{fig:fig-for-err}
\end{figure}

\end{exam}

In the following we will give three toy examples for $\|A^{-1}\|=1$, in which SOCAVEs~\eqref{eq:socave} may have more than one solutions, a unique solution or no solutions. In the following examples, $\gamma = 2$ is used.

\begin{exam}\label{exam:toy-example}
Consider SOCAVEs~\eqref{eq:socave} with
$$
A = \left[\begin{array}{cc}
1& 0\\
0 & -1
\end{array}\right],\quad b = \left[\begin{array}{c}
 0\\
0
\end{array}\right].
$$

Obviously, SOCAVEs~\eqref{eq:socave} has infinitely many solutions for this example. Thus the dynamical model \eqref{eq:dymos} has infinitely many equilibrium points. In fact, $x = (a,b)^\top (a\ge 0, b= 0)$ are equilibrium points of the dynamical model \eqref{eq:dymos}. Specifically, we have the following monotone properties of the solution of \eqref{eq:dymos}.
\begin{itemize}
  \item [(a)] If $x_1\ge |x_2|\ge 0$, then
  \begin{align*}
  \frac{dx_1}{dt}&=0,\\
  \frac{dx_2}{dt}&= -2\gamma x_2 \quad \Rightarrow \quad \left\{\begin{array}{l}
  \frac{dx_2}{dt} \ge 0, \quad \text{if} \quad x_2\le 0,\\
  \frac{dx_2}{dt} < 0, \quad \text{if} \quad x_2> 0.
  \end{array}\right.
  \end{align*}

  \item [(b)] If $-|x_2|< x_1< |x_2|$, then
  \begin{align*}
  \frac{dx_1}{dt}&=\gamma (|x_2| - x_1)>0,\\
  \frac{dx_2}{dt}&= -\gamma (1+\frac{x_1}{|x_2|}) x_2 \quad \Rightarrow  \quad \left\{\begin{array}{l}
  \frac{dx_2}{dt} \ge 0, \quad \text{if} \quad x_2\le 0,\\
  \frac{dx_2}{dt} < 0, \quad \text{if} \quad x_2> 0.
  \end{array}\right.
  \end{align*}

  \item [(c)] If $x_1\le -|x_2|\le 0$, then
  \begin{align*}
  \frac{dx_1}{dt}&=-2\gamma x_1 \ge 0,\\
  \frac{dx_2}{dt}&= 0.
  \end{align*}
\end{itemize}
Figure~\ref{fig:fig-for-toyexam} displays the transient behaviors of $x(t)=\left(x_1(t),x_2(t)\right)^\top$ with $7$ different initial points, from which we find that the trajectories generated by the dynamical system \eqref{eq:dymos} approach to a solution of SOCAVEs~\eqref{eq:socave}. The results in Figure~\ref{fig:fig-for-toyexam} illustrate our statements.

\begin{figure}[t]
{\centering
\begin{tabular}{ccc}
\hspace{-0.5 cm}
\resizebox*{0.50\textwidth}{0.23\textheight}{\includegraphics{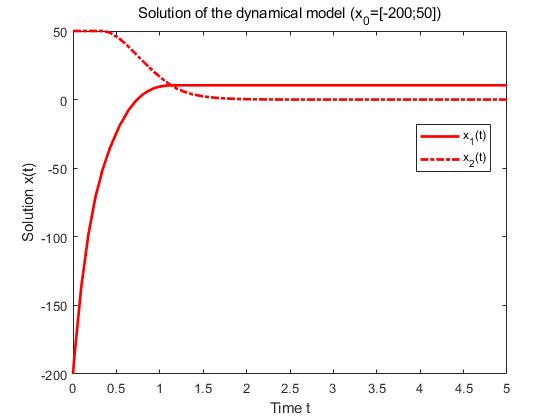}}
& & \hspace{-1 cm}
\resizebox*{0.50\textwidth}{0.23\textheight}{\includegraphics{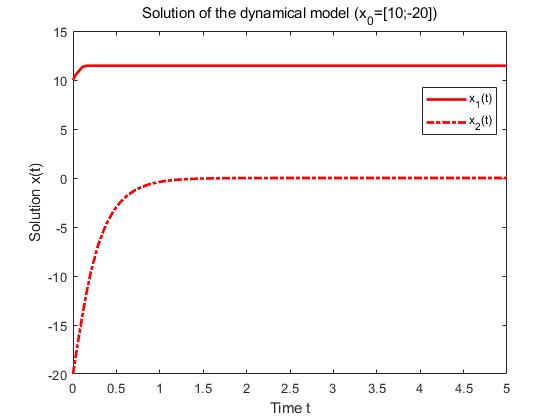}} \vspace{2ex}\\
\hspace{-0.5 cm}
\resizebox*{0.50\textwidth}{0.23\textheight}{\includegraphics{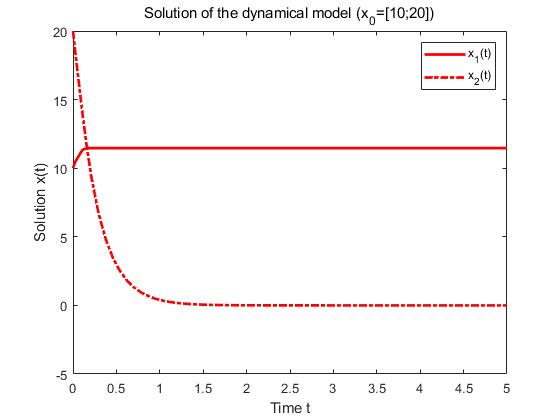}}
& & \hspace{-1 cm}
\resizebox*{0.50\textwidth}{0.23\textheight}{\includegraphics{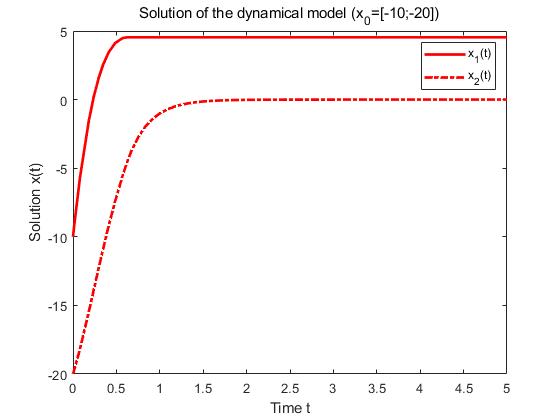}}\vspace{2ex}\\
\hspace{-0.5 cm}
\resizebox*{0.50\textwidth}{0.23\textheight}{\includegraphics{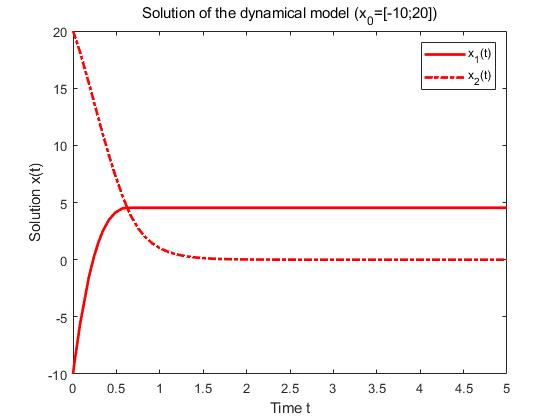}}
& & \hspace{-1 cm}
\resizebox*{0.50\textwidth}{0.23\textheight}{\includegraphics{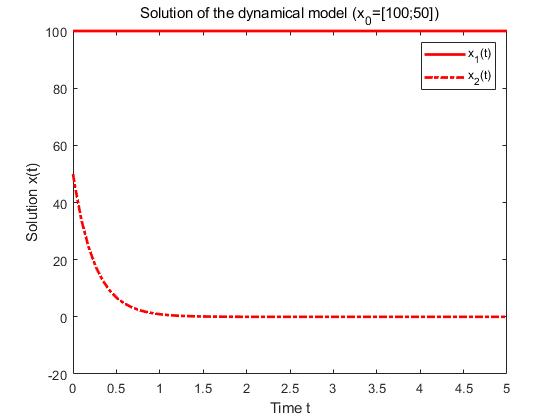}}\vspace{2ex}\\
\hspace{-0.5 cm}
\resizebox*{0.50\textwidth}{0.23\textheight}{\includegraphics{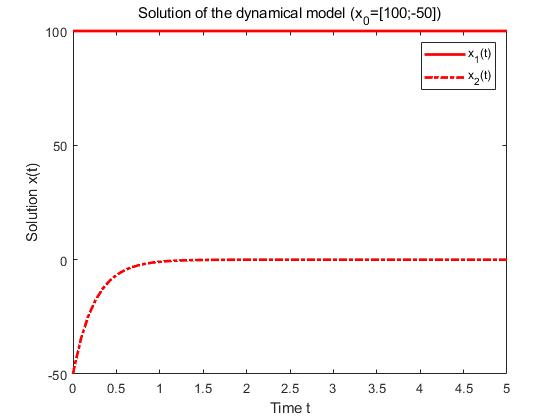}}
& & \hspace{-1 cm}
\end{tabular}\par
}\vspace{-0.15 cm}
\caption{Transient behaviors of $x(t)$ for Example~\ref{exam:toy-example} ($tspan = [0,5]$).}
\label{fig:fig-for-toyexam}
\end{figure}

\end{exam}

\begin{exam}\label{exam:onesolution}
Consider SOCAVEs~\eqref{eq:socave} with
$$
A = \left[\begin{array}{cc}
1& 0\\
0 & -1
\end{array}\right],\quad b = \left[\begin{array}{c}
 -1\\
-1
\end{array}\right].
$$

Obviously, SOCAVEs~\eqref{eq:socave} has a unique solution $x^* = (0, 1)^\top$ for this example. Thus the dynamical model \eqref{eq:dymos} has a unique equilibrium point. In addition, we have the following monotone properties of the solution of \eqref{eq:dymos}.
\begin{itemize}
  \item [(a)] If $x_1\ge |x_2|\ge 0$, then
  \begin{align*}
  \frac{dx_1}{dt}&=-\gamma<0,\\
  \frac{dx_2}{dt}&= -\gamma (-1+2x_2) \quad \Rightarrow \quad \left\{\begin{array}{l}
  \frac{dx_2}{dt} \ge 0, \quad \text{if} \quad x_2\le \frac{1}{2},\\
  \frac{dx_2}{dt} < 0, \quad \text{if} \quad x_2> \frac{1}{2}.
  \end{array}\right.
  \end{align*}

  \item [(b)] If $-|x_2|< x_1< |x_2|$, then
  \begin{align*}
  \frac{dx_1}{dt}&=\gamma (-1+|x_2| - x_1) \quad \Rightarrow \quad \left\{\begin{array}{l}
  \frac{dx_1}{dt} \ge 0, \quad \text{if} \quad x_1\ge 1-|x_2|,\\
  \frac{dx_1}{dt} < 0, \quad \text{if} \quad x_1< 1-|x_2|,
  \end{array}\right.\\
  \frac{dx_2}{dt}&= -\gamma \left[ -1+ (1+\frac{x_1}{|x_2|}) x_2\right] \quad \Rightarrow \quad \left\{\begin{array}{l}
  \frac{dx_2}{dt} \ge 0, \quad \text{if} \quad x_2\le \frac{|x_2|}{x_1+|x_2|},\\
  \frac{dx_2}{dt} < 0, \quad \text{if} \quad x_2> \frac{|x_2|}{x_1+|x_2|}.
  \end{array}\right.
  \end{align*}

  \item [(c)] If $x_1\le -|x_2|\le 0$, then
  \begin{align*}
  \frac{dx_1}{dt}&=\gamma(-1-2 x_1) \quad \Rightarrow \quad\left\{\begin{array}{l}
  \frac{dx_1}{dt} \ge 0, \quad \text{if} \quad x_1\le -\frac{1}{2},\\
  \frac{dx_1}{dt} < 0, \quad \text{if} \quad x_1> -\frac{1}{2},
  \end{array}\right.\\
  \frac{dx_2}{dt}&= \gamma >0.
  \end{align*}
\end{itemize}
Figure~\ref{fig:onesolution} displays the transient behaviors of $x(t)=\left(x_1(t),x_2(t)\right)^\top$ with $8$ different initial points, from which we find that the trajectories generated by the dynamical system \eqref{eq:dymos} approach to the unique solution of SOCAVEs~\eqref{eq:socave}.

\begin{figure}[t]
{\centering
\begin{tabular}{ccc}
\hspace{-0.5 cm}
\resizebox*{0.50\textwidth}{0.23\textheight}{\includegraphics{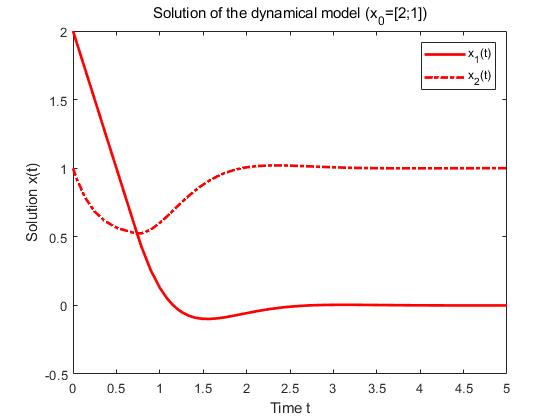}}
& & \hspace{-1 cm}
\resizebox*{0.50\textwidth}{0.23\textheight}{\includegraphics{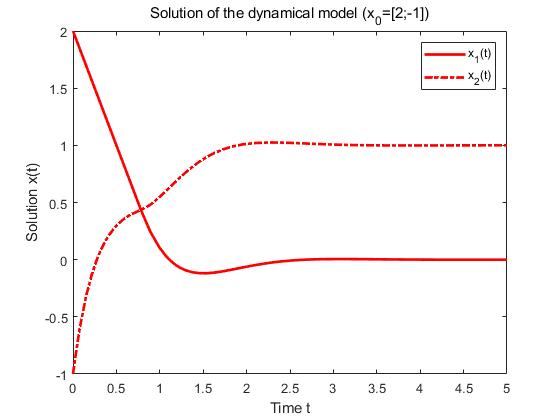}} \vspace{2ex}\\
\hspace{-0.5 cm}
\resizebox*{0.50\textwidth}{0.23\textheight}{\includegraphics{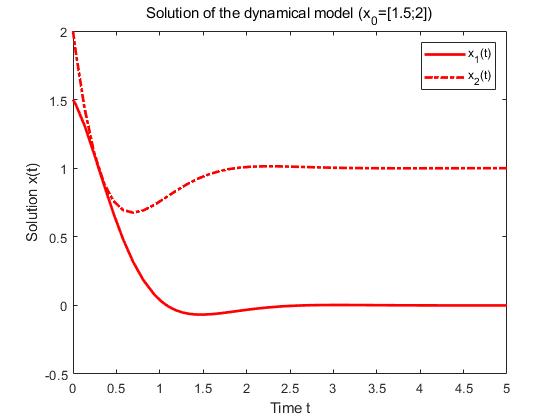}}
& & \hspace{-1 cm}
\resizebox*{0.50\textwidth}{0.23\textheight}{\includegraphics{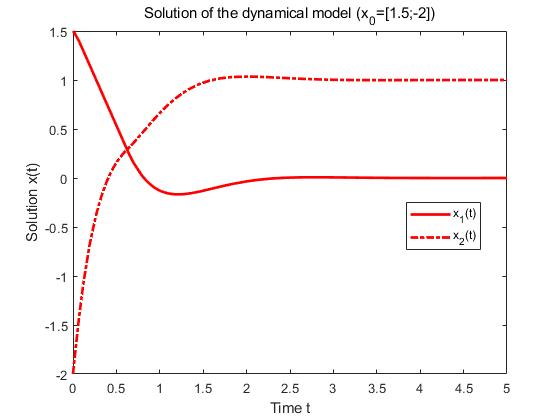}}\vspace{2ex}\\
\hspace{-0.5 cm}
\resizebox*{0.50\textwidth}{0.23\textheight}{\includegraphics{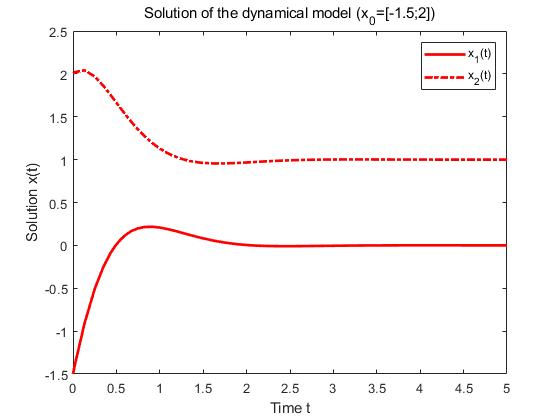}}
& & \hspace{-1 cm}
\resizebox*{0.50\textwidth}{0.23\textheight}{\includegraphics{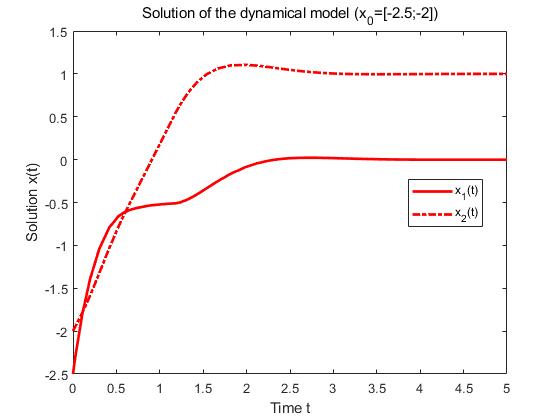}}\vspace{2ex}\\
\hspace{-0.5 cm}
\resizebox*{0.50\textwidth}{0.23\textheight}{\includegraphics{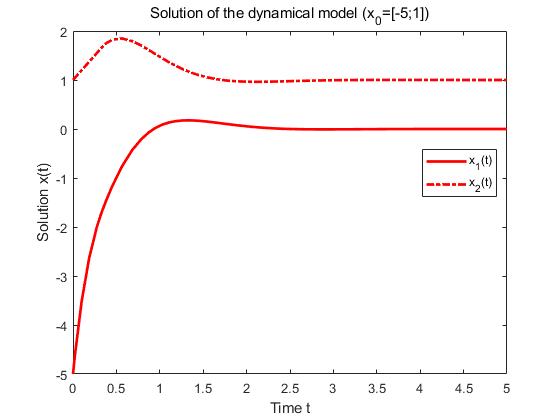}}
& & \hspace{-1 cm}
\resizebox*{0.50\textwidth}{0.23\textheight}{\includegraphics{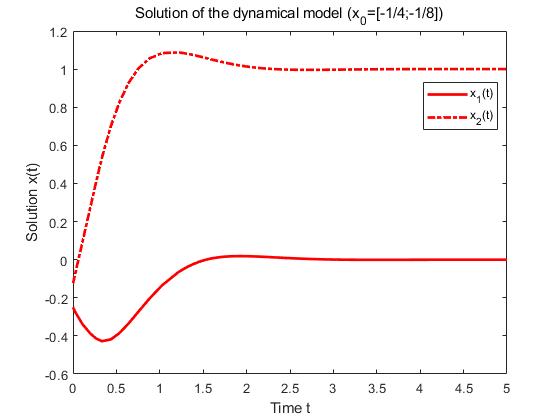}}
\end{tabular}\par
}\vspace{-0.15 cm}
\caption{Transient behaviors of $x(t)$ for Example~\ref{exam:onesolution} ($tspan = [0,5]$).}
\label{fig:onesolution}
\end{figure}

\end{exam}

\begin{exam}\label{exam:nosolution}
Consider SOCAVEs~\eqref{eq:socave} with
$$
A = \left[\begin{array}{cc}
1& 0\\
0 & -1
\end{array}\right],\quad b = \left[\begin{array}{c}
1\\
1
\end{array}\right].
$$

SOCAVEs~\eqref{eq:socave} has no solutions for this example. Thus the dynamical model \eqref{eq:dymos} has no equilibrium points. Indeed, we have
\begin{itemize}
  \item [(a)] If $x_1\ge |x_2|\ge 0$, then
  \begin{align*}
  \frac{dx_1}{dt}&=\gamma > 0,\\
  \frac{dx_2}{dt}&= -\gamma (1+2x_2) \quad \Rightarrow \quad \left\{\begin{array}{l}
  \frac{dx_2}{dt} \ge 0, \quad \text{if} \quad x_2\le -\frac{1}{2},\\
  \frac{dx_2}{dt} < 0, \quad \text{if} \quad x_2> -\frac{1}{2}.
  \end{array}\right.
  \end{align*}

  \item [(b)] If $-|x_2|< x_1< |x_2|$, then
  \begin{align*}
  \frac{dx_1}{dt}&=\gamma (1+|x_2| - x_1)>0,\\
  \frac{dx_2}{dt}&= -\gamma \left[ 1+ (1+\frac{x_1}{|x_2|}) x_2\right] \quad \Rightarrow \quad \left\{\begin{array}{l}
  \frac{dx_2}{dt} \ge 0, \quad \text{if} \quad x_2\le -\frac{|x_2|}{x_1+|x_2|},\\
  \frac{dx_2}{dt} < 0, \quad \text{if} \quad x_2> -\frac{|x_2|}{x_1+|x_2|}.
  \end{array}\right.
  \end{align*}

  \item [(c)] If $x_1\le -|x_2|\le 0$, then
  \begin{align*}
  \frac{dx_1}{dt}&=\gamma (1-2x_1) > 0,\\
  \frac{dx_2}{dt}&= -\gamma <0.
  \end{align*}
\end{itemize}
Thus, for any initial value $x_0$,  at least $x_1(t)$ in the solution of \eqref{eq:dymos} is monotonically increasing. Figure~\ref{fig:nosolution} displays the transient behaviors of $x(t)=\left(x_1(t),x_2(t)\right)^\top$ with $8$ different initial points, which illustrates our claims.

\begin{figure}[t]
{\centering
\begin{tabular}{ccc}
\hspace{-0.5 cm}
\resizebox*{0.50\textwidth}{0.23\textheight}{\includegraphics{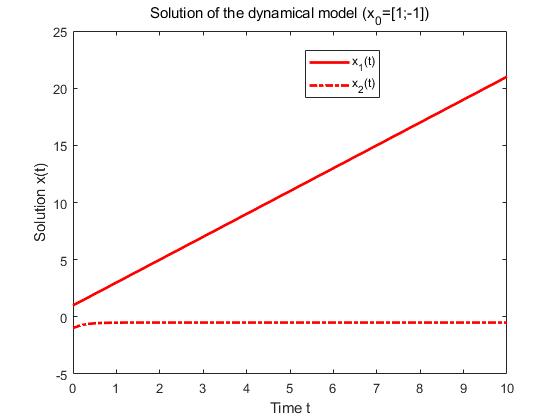}}
& & \hspace{-1 cm}
\resizebox*{0.50\textwidth}{0.23\textheight}{\includegraphics{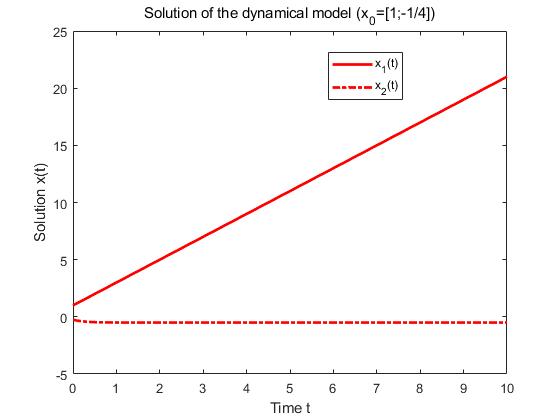}} \vspace{2ex}\\
\hspace{-0.5 cm}
\resizebox*{0.50\textwidth}{0.23\textheight}{\includegraphics{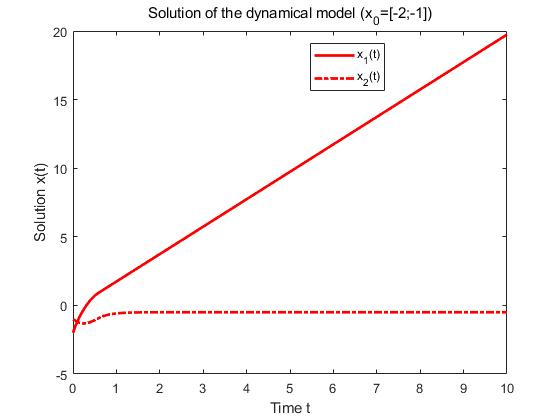}}
& & \hspace{-1 cm}
\resizebox*{0.50\textwidth}{0.23\textheight}{\includegraphics{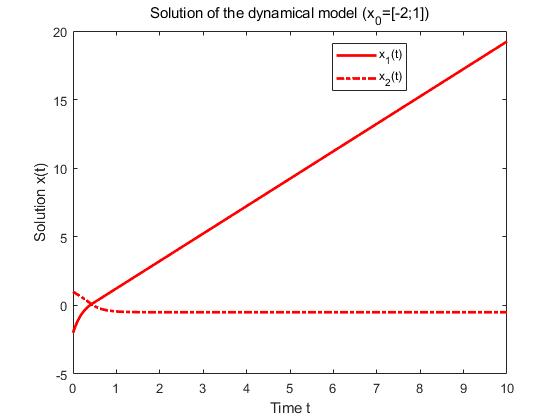}}\vspace{2ex}\\
\hspace{-0.5 cm}
\resizebox*{0.50\textwidth}{0.23\textheight}{\includegraphics{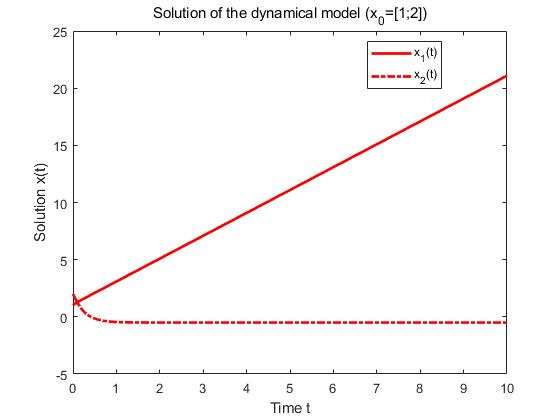}}
& & \hspace{-1 cm}
\resizebox*{0.50\textwidth}{0.23\textheight}{\includegraphics{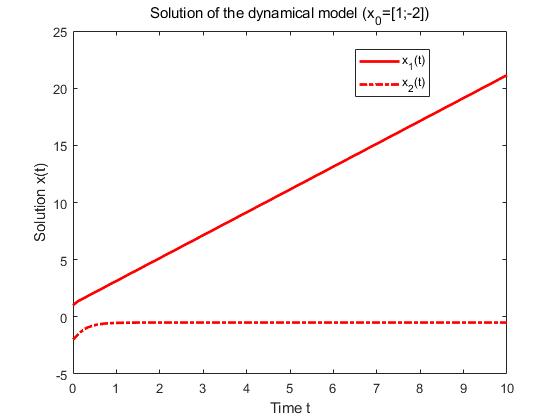}}\vspace{2ex}\\
\hspace{-0.5 cm}
\resizebox*{0.50\textwidth}{0.23\textheight}{\includegraphics{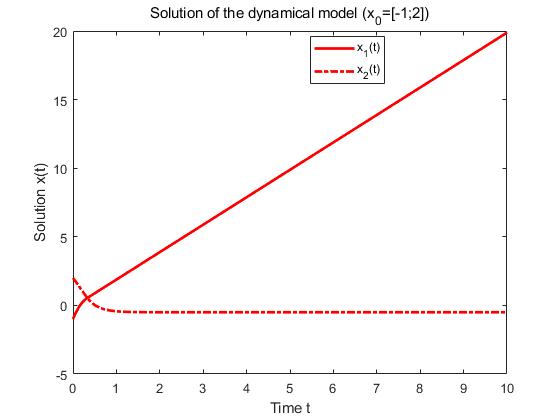}}
& & \hspace{-1 cm}
\resizebox*{0.50\textwidth}{0.23\textheight}{\includegraphics{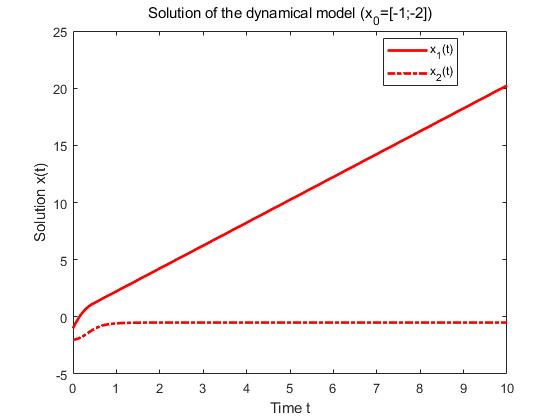}}
\end{tabular}\par
}\vspace{-0.15 cm}
\caption{Transient behaviors of $x(t)$ for Example~\ref{exam:nosolution} ($tspan = [0,10]$).}
\label{fig:nosolution}
\end{figure}

\end{exam}

\section{Brief conclusion}\label{sec:conclusion}
In this paper, a dynamical system is proposed to solve SOCAVEs~\eqref{eq:socave}. The solution of SOCAVEs~\eqref{eq:socave} is equivalent to the equilibrium point of the dynamical model and can be obtained by solving a system of ordinary differential equations. One main feature of our method is that it is inverse-free. The results of
this paper can be considered as extensions of those in \cite{chen2021}.


\end{document}